%% file: main.tex
\documentclass{article}

\usepackage{amsmath,amsthm,amssymb}
\usepackage{enumerate}
\usepackage{graphicx}
\usepackage{cite}

\theoremstyle{plain} 
\newtheorem{thm}{Theorem}[section]

\newtheorem{lem}[thm]{Lemma}
\newtheorem{prop}[thm]{Proposition}
\theoremstyle{definition}
\newtheorem{defn}[thm]{Definition}
\theoremstyle{remark}
\newtheorem{rem}[thm]{Remark}

\newcommand{\R}{\mathbb{R}}
\newcommand{\C}{\mathbb{C}}

\usepackage{times}
\usepackage{verbatim}
\usepackage{framed}
\numberwithin{equation}{section}

\begin{document}

\author{Nick Edelen}
\title{A conservation approach to helicoidal surfaces of constant mean curvature in $\R^3$, $\mathbb{S}^3$ and $\mathbb{H}^3$}
\date{August 12, 2011}
\maketitle

\begin{abstract}
We develop a conservation law for constant mean curvature (CMC) surfaces introduced by Korevaar, Kusner and Solomon \cite{kks}, and provide a converse, so as to characterize CMC surfaces by a conservation law.  We work with `twizzler' construction, which applies a screw-motion to some base curve.  We show that, excluding cylinders, CMC helicoidal surfaces can be completely determined by a first-order ODE of the base curve.  Further, we demonstrate that in $\R^3$ this condition is equivalent to the treadmillsled characterization of helicoidal CMC surfaces given by Perdomo \cite{perdomo}.
\end{abstract}

\section{Introduction}
\input{introduction}

\section{Conservation Law}\label{conservation-section}
\input{conservation}


\section{Twizzlers in $\R^3$, $\mathbb{S}^3$ and $\mathbb{H}^3$}\label{twizzler-section}
\input{specific-twizzlers}

\section{Perdomo's Characterization}

\input{perdomo}

\bibliography{references}{}
\bibliographystyle{alphanum}

\end{document}

%% file: introduction.tex
We study immersed, constant mean curvature (CMC) surfaces which have screw-motion symmetry.  Such surfaces can be described by the \emph{twizzler} construction, which applies a screw-motion to some base curve $\gamma$.  Recently Perdomo \cite{perdomo} characterized CMC twizzlers in $\R^3$ by giving, and geometrically interpreting, a first-order ODE on $\gamma$, using what he called the `treadmillsled coordinates.'

In 1989 Korevaar, Kusner and Solomon \cite{kks} derived (later refined by Kusner \cite{kusner}) a `flux conservation' law of any CMC surface $\Sigma$ immersed in a simply-connected spaceform.  The law describes in essence a conserved quantity on any loop on $\Sigma$ in a given homology class.  When calculating this quantity for certain examples of surfaces, a striking similarity with Perdomo's treadmillsled condition arises.

Here we will extend the result of \cite{kusner} to the converse: given that a certain quantity is conserved over every loop in some homology class of $\Sigma$, then necessarily $\Sigma$ has constant mean curvature.  In each spaceform $\R^3$, $\mathbb{S}^3$ (the $3$-sphere), $\mathbb{H}^3$ ($3$-dimensional hyperbolic space), we use the conservation law to derive a first-order ODE on $\gamma$ for its twizzler to have constant mean curvature.  Further, by considering $\gamma$ parameterized by the angle of normal (the \emph{support} parameterization), we show that the treadmillsled condition of \cite{perdomo} is identical to the conservation law on $\gamma$ in $\R^3$.

This is helpful since the second-order ODE for constant mean curvature is quite difficult to understand.  In $\R^3$ \cite{perdomo} classified geometrically the curves satisfying this first-order ODE, and we hope to find a similar analysis in $\mathbb{S}^3$ and $\mathbb{H}^3$.

I would like to extend my sincerest thanks to Professor Bruce Solomon for supervising, and to the NSF for funding this research as part of the 2011 Indiana University REU program.

\subsection{Preliminaries}

Consider an orientable surface $\Sigma$ embedded in some 3-dimensional, simply-connected spaceform $N$.  Pick $\boldsymbol\nu$ to be the smooth unit normal of $\Sigma$ in $N$.  Let $(\mathbf g_1, \mathbf g_2, \mathbf g_3)$ be any local, orthonormal frame on $N$, and likewise $(\mathbf f_1, \mathbf f_2, \boldsymbol \nu)$ a local, orthonormal adapted frame of $\Sigma$.  For any vector field $\mathbf Y$ on $\Sigma$, we call $\mathbf Y^\top = (\mathbf Y \cdot \boldsymbol \nu) \boldsymbol\nu$ and $\mathbf Y^\bot = \mathbf Y - \mathbf Y^\top$ the normal and tangential components of $\mathbf Y$.

Denote covariant differentiation in $N$ by $\mathrm{D}$.  For any smooth vector field $\mathbf Y$, the divergence in $N$ is given by $\mathrm{DIV}(\mathbf Y) = \sum_1^3 (\mathrm{D}_{\mathbf g_i} \mathbf Y) \cdot \mathbf g_i$, and the divergence on $\Sigma$ by $\mathrm{div}(\mathbf Y) = \mathrm{D}_{\mathbf f_1} \mathbf Y \cdot \mathbf f_1 + \mathrm{D}_{\mathbf f_2} \mathbf Y \cdot \mathbf f_2$.  The gradient and Laplacian on $\Sigma$ are given by $\nabla \mathbf Y = (\mathrm D \mathbf Y)^\bot$ and $\Delta \mathbf Y = \mathrm{div}(\nabla \mathbf Y)$.

All $3$-dimensional integrals are taken with respect to the volume element of $N$; integrals of $2$-dimensions are taken w.r.t. the surface element.

The mean curvature vector of $\Sigma$ is $\mathbf{h} = h \boldsymbol \nu = \Delta \mathbf{x}$, for $\mathbf{x}$ the inclusion mapping.  The mean curvature $h$ is, up to sign, the trace of the second fundamental form.

Let $(\mathbf e_1, \cdots, \mathbf e_n)$ denote the standard orthonormal frame of $\R^n$.  In $\R^3$, we will often (implicitly) identify $\mathrm{span}(\mathbf e_1, \mathbf e_2)$ with $\C$ by $a \mathbf e_1 + b \mathbf e_2 \leftrightarrow a + i b$.  Likewise, we will identity $\R^4$ with $\C \times \C$.

%% file: conservation.tex
We loosely follow \cite{kks}.  Let $\Sigma$ be a connected, orientable surface embedded in a 3-dimensional, simply-connected spaceform $N$ (i.e. effectively $\R^3$, $\mathbb{S}^3$ or $\mathbb{H}^3$).  Identify the Killing fields on $N$ with the Lie algebra $\mathfrak{g}$ of $N$'s isometry group.  Observe that the Killing vectors span each tangent space on $N$.

Let $\Gamma_1$ and $\Gamma_2$ be two smooth, homologous $1$-cycles in $\Sigma$ bounding a compact subset $S \subset \Sigma$.  As $H_1(N) = 0$, we can take $K_1$ and $K_2$ to be any smooth $2$-chains so that $\partial K_i = -\Gamma_i$.  Then, since $H_2(N) = 0$, there is a $3$-chain $U \subset N$ with piecewise smooth boundary $\partial U = S + K_1 - K_2$.

We write $\boldsymbol\nu$ for the unit normal on any $2$-chain in $N$; likewise, denote the unit conormal on any $1$-chain in $\Sigma$ by $\boldsymbol\eta$.

Pick a Killing field $\mathbf Y \in \mathfrak g$.  It is easily found that the variation of volume $|U|$ along $\mathbf Y$ is
\[
\delta_{\mathbf Y}(|U|) = \int_U \mathrm{DIV}(\mathbf Y) = \int_{K_1-K_2} \mathbf Y \cdot \nu + \int_S \mathbf Y \cdot \nu
\]

Similarly, the variation of area $|S|$ along $\mathbf Y$ is known to be \cite{simon}:
\begin{align*}
\delta_{\mathbf Y}(|S|) = \int_S \mathrm{div}(\mathbf Y) & = \int_S \mathrm{div}({\mathbf Y}^\top) + \mathrm{div}({\mathbf Y}^\bot) \\
 & = \int_{\partial S} \mathbf Y \cdot \boldsymbol \eta + \int_S \mathbf Y \cdot \boldsymbol \nu \, \mathrm{div}(\boldsymbol \nu) \\
 & = \int_{\Gamma_1 - \Gamma_2} \mathbf Y \cdot \boldsymbol \eta + \int_S \mathbf Y \cdot h \boldsymbol \nu
\end{align*}
having used Stokes' theorem, and observing that $\mathrm{div}(\boldsymbol \nu) = \mathrm{trace}(\nabla \boldsymbol \nu) = h$.

Combining these two calculations, we obtain the \emph{first variation formula}:
\begin{align}
0 &= \delta_{\mathbf Y}(|S| - H|U|) \notag\\
&= \int_{\Gamma_1 - \Gamma_2} \mathbf Y \cdot \boldsymbol \eta - H \int_{K_1-K_2} \mathbf Y \cdot \boldsymbol \nu + \int_S (h - H) \mathbf Y \cdot \boldsymbol \nu \label{first-variation-formula}
\end{align}
The first equality is a direct consequence of $\mathbf Y$ being a Killing field.

Our main theorem arises naturally from relation \eqref{first-variation-formula}.

\begin{thm}\label{conservation-theorem}
Using the above notation, if $\Sigma$ has constant mean curvature $H$ then there is a linear function $\omega : H_1(\Sigma) \rightarrow \mathfrak{g}^*$ defined by
\begin{equation}\label{conservation-formula}
\omega([\Gamma])(\mathbf Y) = \oint_\Gamma \mathbf Y\cdot \boldsymbol \eta - H \iint_K \mathbf Y \cdot \boldsymbol \nu
\end{equation}
where $K$ is any smooth $2$-chain with $\partial K = -\Gamma$.

Conversely, if for any homology class $[\Gamma] \in H_1(\Sigma)$, $\omega([\Gamma])$ as given by \eqref{conservation-formula} is well-defined, then $\Sigma$ has constant mean curvature.
\end{thm}

\begin{proof}
If $h = H$ everywhere on $\Sigma$, then the first variation formula \eqref{first-variation-formula} reduces to
\[
0 = \oint_{\Gamma_1 - \Gamma_2} \mathbf Y \cdot \boldsymbol \eta - H \iint_{K_1-K_2} \mathbf Y \cdot \boldsymbol \nu
\]
and as our choice of $\Gamma_i$ is arbitrary, we can fix $\Gamma_1$, and immediately observe that for any $\Gamma_2$,
\[
\oint_{\Gamma_2} \mathbf Y \cdot \boldsymbol \eta - H \iint_{K_2} \mathbf Y \cdot \boldsymbol \nu = \oint_{\Gamma_1} \mathbf Y \cdot \boldsymbol \eta - H \iint_{K_1} \mathbf Y \cdot \boldsymbol \nu = \mathrm{constant}
\]

Conversely, suppose that $\omega$ is well-defined for the null homology class $[0] \in H_1(\Sigma)$.  Then by the first variation formula,
\begin{equation}\label{mean-curvature-0-integral}
0 = \iint_S (h - H) \mathbf Y \cdot \boldsymbol \nu
\end{equation}
for any compact $S \subset \Sigma$ with smooth boundary.

Suppose, towards a contradiction, there is a $p \in \Sigma$ with $h(p) \neq H$; wlog suppose $h(p) > H$, and hence $h(p) > H$ in some open neighborhood $S_1$ of $p$.  Then we can choose a $\mathbf Y$ so that $\mathbf Y(p) \cdot \boldsymbol \nu(p) > 0$, an $\epsilon > 0$, and a neighborhood $S_2 \subset S_1$ of $p$ so that $\mathbf Y \cdot \boldsymbol \nu > \varepsilon$ on $S_2$.  Then for any sufficiently small ball $B$ centered at $p$, with $B \cap \Sigma \subset S_2$, 
\[
0 = \iint_{B \cap \Sigma} (h - H) \mathbf Y \cdot \boldsymbol \nu > 0
\]
This contradiction shows that $h \equiv H$ on $\Sigma$.

More generally, if $\omega$ is well-defined for an arbitrary homology class $[\Gamma] \in H_1(\Sigma)$, we have that $\omega([\Gamma + \Gamma_0]) = \omega([\Gamma])$ for any null-homologous $\Gamma_0$.  Thus, by linearity of $\omega$, $\omega([\Gamma_0]) = 0$, and we have already shown this forces $\Sigma$ to have constant mean curvature.
\end{proof}

\begin{rem}\label{conservation-remark}
Embedding of $\Sigma$ is not essential in the theorem above, which can be readily generalized to immersed surfaces.

\end{rem}

\begin{rem}
In a more general setting, the invariant $\omega$ actually lives in the relative 2-homology $H_2(N, \Sigma \cup B)$, where $B$ is a basis for 1-homology of $N$.  Bruce Solomon and the author will generalize theorem \ref{conservation-theorem} to non-trivial homologies in a joint paper, to appear soon.
\end{rem}

%% file: specific-twizzlers.tex
In this section we introduce twizzlers to explicitly parameterize helicoidal surfaces in terms of a base curve.  Using the ideas presented by theorem \ref{conservation-theorem}, we then derive a first-order ODE on the base curve to characterize CMC twizzlers.

We will only lay out in full the proof of twizzlers in $\R^3$, as it readily generalizes to $\mathbb{S}^3$ and $\mathbb{H}^3$.

\subsection{Case: $\R^3$}

\begin{defn}
Let $\gamma : I \rightarrow \C$ be an immersed $C^2$ curve on the interval $I \subset \R$, and $m \in \R_+$.  Then the \emph{twizzler of $\gamma$ with pitch m} in $\R^3$ is the surface $T$ parameterized by
\[
T(u, v) = e^{i v}\gamma(u) + m v \mathbf{e_3} \quad (u, v) \in I \times \R
\]
which we may also reference with the pair $\langle \gamma, m \rangle$.
\end{defn}

$T$ is always immersed, orientable, and connected.  Holding $u = u_0$ constant, we call the curve $T(u_0, v)$ a \emph{helix} of $T$.  $T$ has discrete translational symmetry along the $z$-axis, described by the group $G$ with action $g \cdot \mathbf{r} = \mathbf{r} + 2\pi m \mathbf{e_3}$.  Denote the quotient surface $T / G$ by $\hat{T}$, and observe that the helices of $T$ are smooth loops in $\hat{T}$.

Using the language of section \ref{conservation-section}, we would like to take $\Gamma$ as a helix of $T$, and $N$ as $\R^3/G$.  Complications arise, however, since helices are not null-homologous in $\R^3/G$.  We bypass this issue by using the $z$-axis as a `reference' $1$-cycle, homologous to the helices in $\R^3/G$, to construct $2$-chains with a common boundary.

As suggested by our remark \ref{conservation-remark}, we will explicitly prove that theorem \ref{conservation-theorem} applies although $T$ is not embedded.

\begin{defn}
The \emph{shaving of $\hat{T}$ at $u_0$} is the surface $\mathcal{T}[u_0]$ parameterized by
\[
\mathcal{T}[u_0](v, t) = t e^{i v} \gamma(u_0) + m v e_3 \quad (v, t) \in [0, 2\pi] \times [0, 1]
\]
we can consider $\mathcal{T}[-]$ a bijection between $I$ and shavings on $\hat{T}$.
\end{defn}

\begin{thm}
The twizzler $T \equiv \langle \gamma, m \rangle$ in $\R^3$ has constant mean curvature iff there are constants $H$ and $C$ so that
\begin{equation}\label{first-integral-r3-twizzler}
C = \frac{2\pi}{\sqrt{g}} m\gamma' \cdot i \gamma - H\pi |\gamma|^2
\end{equation}
where $\sqrt{g} = \sqrt{(T_u \cdot T_u)(T_v \cdot T_v) - (T_u \cdot T_v)^2}$ is the area density on $T$.

Further, if $T$ is not a cylinder, $H$ is the mean curvature of $T$.
\end{thm}

\begin{proof}
Take $\mathbf{Y} = \mathbf{e_3}$, generating translation along the $z$-axis; $\mathbf{Y}$ then descends to $\R^3/G$.  Assume $T$ has constant mean curvature $H$, and pick an interval $J \subset I$ so that $S = T(J \times [0, 2\pi])$ is embedded.  For $i = 1, 2$, let $\Gamma_i \subset S$ be the helix at any point $u_i \in J$, and $K_i \equiv \mathcal{T}[u_i]$.  Then $S + K_1 - K_2$ bounds a compact volume in $\R^3/G$, and the first variation formula \eqref{first-variation-formula} holds.  As in the proof of theorem \ref{conservation-theorem}, we deduce the existence of a constant $C$ so that, for any shaving $\mathcal{T}[u]$, $u \in J$, 
\begin{equation}\label{conservation-formula-r3-twizzler}
C = \oint_{\partial \mathcal{T}[u]} \mathbf{Y}\cdot \boldsymbol \eta - H \iint_{\mathcal{T}[u]} \mathbf{Y} \cdot \boldsymbol \nu
\end{equation}

Note that if $T$ is a cylinder, $\iint_S \boldsymbol \nu$ vanishes for any $S \subset \Sigma$ bounded by helices, so for any $H$ we can find a $C$ so that \eqref{conservation-formula-r3-twizzler} holds.  We then evaluate \eqref{conservation-formula-r3-twizzler} explicitly to give relation \eqref{first-integral-r3-twizzler}.

Using Gram-Schmidt, we calculate that
\begin{align*}
\boldsymbol\eta \cdot \mathbf Y &= \frac{\sqrt{T_v \cdot T_v}}{\sqrt{(T_u \cdot T_u)(T_v \cdot T_v) - (T_u \cdot T_v)^2}} \left(T_u - \frac{T_u \cdot T_v}{T_v \cdot T_v} T_v\right) \cdot \mathbf Y\\
 &= -\frac{\sqrt{|\gamma|^2 + m^2}}{\sqrt{g}} \frac{\gamma' \cdot i\gamma}{|\gamma|^2+m^2} m
\end{align*}
so that
\begin{align*}
\oint_{\partial S} \mathbf Y \cdot \boldsymbol \eta &= \frac{1}{\sqrt{g}} \int_0^{2\pi} - \frac{m \gamma' \cdot i\gamma}{|\gamma|^2+m^2} (|\gamma|^2 + m^2) \,\mathrm d v \\
 &= -\frac{2\pi}{\sqrt{g}} m\gamma' \cdot i\gamma
\end{align*}.

On the shaving, 
\[
\boldsymbol \nu \,\mathrm d S = S_v \wedge S_t = i\gamma e^{i v} - t|\gamma|^2\mathbf{e_3}
\]
Note order of cross product: we must be consistent with the requirement $\partial K = -\Gamma$.  The second integral becomes
\[
\iint_S \mathbf Y \cdot \boldsymbol \nu = -\int_0^{2\pi}\int_0^1 t|\gamma|^2 \,\mathrm d t \,\mathrm d v = -\pi|\gamma|^2
\]

For each point $u$ we can find a neighborhood $u \in J_u \subset I$ so that $T(J_u \times [0, 2\pi])$ is embedded, and hence a constant $C_u$ satisfying \eqref{first-integral-r3-twizzler} on $T(J_u \times [0, 2\pi])$.  By compactness, any closed interval $[a, b] \subset I$ is covered by finitely many $J_u$, implying that $C_a = C_b$.  We deduce that $C = C_a = C_b$ satisfies \eqref{first-integral-r3-twizzler} everywhere on $T$.

Conversely, given that \eqref{first-integral-r3-twizzler} -- and hence the conservation formula \eqref{conservation-formula} -- holds for every shaving, the first variation formula \eqref{first-variation-formula} gives
\begin{align}
0 &= \iint_{\mathbf{T}(J \times [0, 2\pi])} (h - H) \mathbf{Y} \cdot \boldsymbol \nu \notag \\
 &= 2\pi \int_{\gamma(J)} (h - H) \mathbf{Y} \cdot \boldsymbol \nu \sqrt{|\gamma|^2 + m^2} \label{non-zero-shaving-integral-r3}
\end{align}
whenever $J \subset I$ is sufficiently small for $T(J \times [0, 2\pi])$ to be embedded.

The subset $I^* \subset I$ defined by $\mathbf{Y} \cdot \boldsymbol \nu \neq 0$ is open in $I$.  Suppose $h(p) \neq H$ at some point $p \in I^*$, then there is a ball $B_\varepsilon \ni p$ in $I^*$ on which $h - H$ has a fixed sign.  But likewise $\mathbf{Y}\cdot\boldsymbol \nu \neq 0$ in $B_\varepsilon$, so the integral \eqref{non-zero-shaving-integral-r3} cannot vanish, a contradiction.  It follows that $h \equiv H$ on $I^*$.

If $I^*$ is also dense in $I$, then $h = H$ everywhere.  If $I^*$ is not dense, there is a maximal interval $I_0 \subset I - I^*$ on which $\mathbf{Y} \cdot \boldsymbol \nu = 0$.  Then necessarily, $\gamma' \bot \gamma$ on $I_0$, and hence $T(I_0 \times [0, 2\pi])$ is a segment of a cylinder having some constant mean curvature $H_0$.

If $I_0 = I$ then $T$ is entirely a cylinder, with mean curvature $H_0$.  Otherwise there is a sequence of points in $I^*$ approaching an end-point of $I_0$, implying by continuity of $h$ that $H = H_0$.
\end{proof}

\subsection{Case: $\mathbb{S}^3$}

Embed $\mathbb{S}^3$ in $\R^4$ by equipping the submanifold $\{x \in \R^4 | x \cdot x = 1\}$ with the induced metric.

\begin{defn}
Let $(\gamma, f) : I \rightarrow \C \times \R \subset \mathbb{S}^3$ be a curve on the interval $I \subset \R$, so that $|\gamma|^2 + f^2 = 1$.  Pick an $m \in \R_+$.  Define the twizzler $T : I \times \R \rightarrow \C \times \C \cong \R^4$ in $\mathbb{S}^3$ to be the surface parameterized by
\[
T(u, v) = (e^{i v}\gamma(u), e^{i m v} f(u))
\]
\end{defn}

\begin{defn}
For a twizzler $T$ in $\mathbb{S}^3$, define the shaving at $u_0 \in I$ to be the surface parameterized by
\[
\mathcal{T}[u_0](v, t) = (e^{i v}\frac{\gamma(u_0)}{|\gamma(u_0)|}\sin t, e^{i m v}\cos t) \quad (v, t) \in [0, 2\pi] \times [0, \cos^{-1} f(u_0)]
\]
\end{defn}

\begin{thm}
The twizzler $T$ in $\mathbb{S}^3$ has constant mean curvature iff there are constants $H$ and $C$ so that, for all values of $\gamma$
\begin{equation}
C = \frac{2\pi}{\sqrt{g}} m f^2 (\gamma' \cdot i\gamma) - H\pi|\gamma|^2
\end{equation}
where $g = (T_u \cdot T_u)(T_v \cdot T_v) - (T_u \cdot T_v)^2$.

Further, if $T$ is not a torus, $H$ is the mean curvature of $T$.
\end{thm}

\begin{proof}
Pick $\mathbf{Y}(z, w) = (0, iw)$.  Then the proof is precisely the same as in $\R^3$, only considering the $\mathbb{S}^3$ equivalent of the cylinder to be the torus.  Torii in $\mathbb{S}^3$ are defined by the parameterization $T(u, v) = (\cos  e^{i u}, \sin t e^{i v})$, for some constant $\xi \in \R$.  All $\mathbb{S}^3$ torii have constant mean curvature, as every point $T(u, v)$ can be mapped to $T(0, 0)$ by ambient isometries.
\end{proof}

\subsection{Case: $\mathbb{H}^3$}

We work with the Lorentz model of $\mathbb{H}^3$, as follows.  Define the quadratic form $Q = diag(-1, -1, -1, 1)$.  Then $\mathbb{H}^3 \cong \{x \in \R^4 | <x, x> = 1\}$ with inner product $<x, y> = x^\top y$.  Further, for some non-zero real $m$, define the mapping $B_m : \R \rightarrow SO(1, 1)$ by
\[
B_m(v) = \begin{pmatrix}
\cosh mv & \sinh mv \\
\sinh mv & \cosh mv
\end{pmatrix}
\]

\begin{defn}
Let $(\gamma, f) : I \rightarrow \C \times \R$ be a curve on the interval $I \subset \R$, so that $f^2 - |\gamma|^2 = 1$.  Pick $m \in \R_+$ .  Define the twizzler $T : I \times [-\infty, \infty] \rightarrow \C \times \C$ in $\mathbb{H}^3$ by the parameterization
\[
T(u, v) = (e^{i v} \gamma(u), B_m(v) i f(u) )
\]
\end{defn}

Define a cylinder in hyperbolic space to be the twizzler of $(\gamma(u), f(u)) = (a e^{i u}, b)$, for $a, b \in \R$ satisfying $b^2 - a^2 = 1$.

\begin{defn}
For a twizzler $T$ in $\mathbb{H}^3$, define the shaving at $u_0 \in I$ to be the surface $S$ parameterized by
\[
\mathcal{T}[u_0](v, t) = (e^{i v} \frac{\gamma(u_0)}{|\gamma(u_0)|}\sinh t, B_m(v) i\cosh t) \quad (v, t) \in [0, 2\pi] \times [0, \cosh^{-1} f(u_0)]
\]
\end{defn}

\begin{thm}
The twizzler $T$ in $\mathbb{H}^3$ has constant mean curvature iff there are constants $H$ and $C$ so that, for all values of $\gamma$
\begin{equation}
C = \frac{2\pi}{\sqrt{g}}m f^2 <\gamma', i \gamma> + H\pi|\gamma|^2
\end{equation}
where $g = <T_u , T_u><T_v , T_v> + <T_u , T_v>^2$.

Further, if $T$ is not a cylinder, then $H$ is the mean curvature of $T$.
\end{thm}

\begin{proof}
Take $\mathbf{Y}(z, w) = (0, \begin{pmatrix} 0 & 1 \\ 1 & 0 \end{pmatrix} w)$.  The proof now follows as in the $\R^3$ case.
\end{proof}

%% file: perdomo.tex
An alternate characterization of helicoidal, constant mean curvature surfaces in $\R^3$ is given by Perdomo \cite{perdomo}.  A first integral for the second-order constant mean curvature ODE is interpreted by a kinetic condition called the \emph{treadmillsled}.  We shall use a special parameterization (by angle of normal) to better relate his condition with our conservation law.

\subsection{The treadmillsled}

Intuitively, the treadmillsled is a variation of a roulette -- imagine rolling a curve along a line, while simultaneously moving the line in the opposite direction, so that the curve's point of contact stays in one place.  By tracing out the path of a point fixed relative to the curve, one obtains the treadmillsled.  We shall consider a slight generalization, allowing for any proportion $\ell$ of movement of the line: when $\ell = 0$, the line doesn't move, and we have a roulette; when $\ell = 1$, the line matches the curve's speed, yielding the treadmillsled.

\begin{defn}
The \emph{$\ell$-treadmill} of a $C^2$ curve $\gamma: I \rightarrow \C$ is defined by $\sigma_\ell[\gamma] = (1 - \ell)s - \frac{1}{v}\gamma' \bar{\gamma}$, where $s$ is the arc-length of $\gamma$, and $v$ the speed.
\end{defn}

By construction, $\sigma_\ell$ is independent of parameterization of $\gamma$.  We write $\tau$ for Perdomo's treadmillsled, which is the same as our $\sigma_1$.

\begin{prop}\label{treadmillsled_proposition}
A curve $(x, y): I \rightarrow \C$ is the $\ell$-treadmill of a curve $\gamma$ iff it satisfies the differential equations
\begin{equation}\label{sigma_diffeqs}
\begin{alignedat}{2}
\frac{1}{s'}x' & = -\ell + k y \\
\frac{1}{s'}y' & = (1 - \ell)ks - k x
\end{alignedat}
\end{equation}
with $s: I \rightarrow \R_+$ is strictly increasing, and $k: I \rightarrow \R$.

Further, up to rotations, $\sigma_\ell$ maps $C^2$ curves injectively to $C^1$ curves.
\end{prop}

\begin{proof}
Given $\gamma$, verifying \eqref{sigma_diffeqs} if $(x, y) = \sigma_\ell[\gamma]$ is a simple calculation.  $k$ is then the curvature of $\gamma$, and $s$ the arc-length.

Conversely, given equation \eqref{sigma_diffeqs}, then by the fundamental theorem of curves, there is a planar curve $\gamma(t)$ having curvature $k$ and arc-length $s$.  The further condition that $\sigma_\ell[\gamma](0) = (x, y)(0)$ fixes the orientation of $\gamma$ with respect to the origin, i.e. up to rotation.  The curve $\gamma$ is well-defined in the sense that, if $(x_1, y_1) = (x, y) \circ f$ is a reparameterization of $(x, y)$, then $\sigma_\ell[\gamma \circ f] = (x_1, y_1)$.
\end{proof}

Observe that shifting the `starting position' of $\gamma$ will effectively translate $\sigma_\ell[\gamma]$.  Formally, if $I = (a, b] \cup (b, c)$, then $\sigma_\ell[\gamma]|_{(b, c)} = (1 - \ell)s(b) + \sigma_\ell[\gamma|_{(b, c)}]$.  This is true except for the special case $\tau$ (i.e. $\ell = 1$); since no $s$ term is present, $\tau$ is determined without the ambiguity of `starting position'.  Further, $\tau[\gamma](p)$ is determined by only $\gamma(p)$ and $\gamma'(p)$.  This fact, and that $\tau$ is continuous, immediately gives the following lemma.

\begin{lem}\label{piecewise_tau}
Given a smooth curve $\gamma : I \rightarrow \C$, knowing the values of $\gamma$ and $\gamma'$ on some dense subset of $I$ is sufficient to completely determine $\tau[\gamma]$.
\end{lem}

\subsection{Support parameterization}
Our key angle of attack in relating the ODE's of \cite{perdomo} and the conservation law lies in our choice of parameterization.  We will parameterize in the (cumulative) angle of normal, called the \emph{support parameterization}.  We shall lay out the relavant machinery below.

\begin{defn}
A curve is \emph{strictly convex} if the curvature never vanishes.
\end{defn}

\begin{lem}
Every strictly convex curve can be support-parameterized.
\end{lem}

\begin{proof}
Consider a curve $\gamma : I = (a, b) \rightarrow \C$ with curvature $k$ and speed $v$.  Without loss of generality suppose $k < 0 \,\, \forall t \in I$.  Let $n$ be the normal of the curve $\gamma$, and $\theta = \arg(n)$ be the normal angle.  Then $n := \frac{i}{v}\gamma'$.

Define the cumulative normal angle by $\Theta : I \rightarrow \R$ by $\Theta(t) = \int_a^t \mathrm{d}\theta$.  Since $\theta' = -vk > 0$, $\Theta$ is strictly increasing.  Therefore there is a an inverse $\Theta^{-1} : \Theta(I) \rightarrow I$.
\end{proof}

There is a nice form for the support parameterization.  Let $\theta$ be the normal angle of $\gamma$.  We can then write $\gamma(\theta) = (q + i r)e^{\imath \theta}$, for some functions $q, r : \Theta(I) \rightarrow \R$.

So that $\theta$ is indeed the angle of the normal, we require that $\theta = \arg(n) = \arg(i\gamma')$, imposing the condition $q' = r$.  We have then $\gamma = (q + iq')e^{\imath \theta}$, and $\gamma'' = (q + q'')ie^{i\theta}$.  As long as $q + q'' > 0$, this is a valid parameterization of $\gamma$, by the above lemma.  (Conversely, every $q$ satisfying $q + q'' > 0$ is the support function for some convex curve.)

The function $q : \Theta(I) \rightarrow \R$ is called the \emph{support function of $\gamma$}.  The advantage in this choice of parameterization becomes clear from the following.

\begin{prop}\label{support_tau}
If strictly convex curve $\gamma$ has support function $q$, then
\[ \tau[\gamma] = -q' - i q \]
\end{prop}

\begin{proof}
From the definition of $\sigma_1 \equiv \tau$, and the above comments, we have
\begin{align*}
\tau[\gamma](\theta) &= -\frac{1}{q + q''} (q + q'') i e^{i \theta} (q - i q') e^{-i \theta} \\
 &= -q' - iq
\end{align*}
\end{proof}

\subsection{Twizzlers of constant mean curvature}

We will quote Perdomo's theorem without proof, but will provide a sketch of its origin.

\begin{thm}\label{perdomo_cmc_theorem}
(Perdomo) The twizzler $T \equiv \langle \gamma, m \rangle$ has constant mean curvature $H$ iff: $T$ is a cylinder of radius $-\frac{1}{2H}$, or $\tau[\gamma]$ satisfies
\begin{equation}\label{perdomo_cmc_formula}
H(x^2 + y^2) - \frac{2my}{\sqrt{m^2 + x^2}} = M
\end{equation}
for some constant $M \geq -\frac{1}{H}$.
\end{thm}

\begin{rem}
The above theorem arises directly as a first-integral of the second-order ODE condition for CMC twizzlers.  Explicitly, suppose $\gamma$ has a support parameterization.  Then express $\tau[\gamma]$ in terms of the support function $q$ of $\gamma$, and differentiate \eqref{perdomo_cmc_formula}.  We obtain
\[
\frac{H}{2m}(2q'q + 2q''q') = \frac{-q'}{(m^2 + q'^2)^{1/2}} + \frac{q''q'q}{(m^2 + q'^2)^{3/2}}
\]
which, if $q' \neq 0$ (i.e. $T$ is not a cylinder), simplifies to the canonical expression for mean curvature $H$ of a twizzler, expressed in terms of a support function $q$:
\[
H = \frac{-m}{v(m^2 + q'^2)^{3/2}}(m^2+q'^2-q q'')
\]
\end{rem}

\subsection{Equivalence}

An equivalance can be directly established between the two treadmillsled and conservation law characterizations.  This equivalence is clear if $T$ is the helicoid, as both relations \eqref{first-integral-r3-twizzler} and \eqref{perdomo_cmc_formula} reduce to $0$.  Otherwise we need the following lemma.

\begin{lem}\label{line-segment-lemma}
An xy-slice of a CMC twizzler contains a line segment iff it is a line.
\end{lem}

\begin{proof}
Immediate from the fact the only ruled, non-planar CMC surface is the helicoid, and the well-known property that CMC surfaces are real analytic.
\end{proof}

Then if $T$ is not a helicoid, using lemmas \ref{line-segment-lemma} and \ref{piecewise_tau}, we can (piecewise) support parameterize $\gamma$ of the twizzler $\langle \gamma, m \rangle$, using support function $q$.  The conservation law \eqref{first-integral-r3-twizzler}, written in terms of $q$, is precisely equation \eqref{perdomo_cmc_formula}:
\begin{align*}
C
 &= \frac{2\pi}{\sqrt{g}} m\gamma' \cdot i \gamma - H\pi |\gamma|^2 \\
 &= \frac{2\pi m q (q + q'')}{\sqrt{(q + q'')^2(q^2 + q'^2 + m^2) - (q + q'')^2 q^2}} - H \pi(q^2 + q'^2) \\
 &= \frac{2\pi}{\sqrt{m^2 + q'^2}} m q - H\pi(q^2 + q'^2) \\
 &= -\pi M
\end{align*}

Thus, our theorem \ref{conservation-theorem} can be expressed in the language of Perdomo's treadmillsled, and conversely Perdomo's theorem \ref{perdomo_cmc_theorem} can be rewritten as a conservation law.

\begin{thm}
If $C$ is the constant from theorem \ref{conservation-theorem}, and $M$ the constant from theorem \ref{perdomo_cmc_theorem}, then $C = -\pi M$.
\end{thm}

%% file: main.bbl
\begin{thebibliography}{KKS}

\bibitem[KKS]{kks}
N.~Korevaar, R.~Kusner, and B.~Solomon.
\newblock The structure of complete embedded surfaces with constant mean
  curvature.
\newblock {\em Journal of Differential Geometry}, 30:465--503, 1989.

\bibitem[Kus]{kusner}
R.~Kusner.
\newblock Bubbles, conservation laws and balanced diagrams.
\newblock In R.~Finn P.~Concus and D.~Hoffman, editors, {\em Geometric Analysis
  and Computer Graphics}. Springer-Verlage, 1990.

\bibitem[Per]{perdomo}
O.~Perdomo.
\newblock A dynamical interpretation of the profile curve of cmc twizzler
  surfaces.
\newblock ArXiv preprint \#1001.5198.

\bibitem[Sim]{simon}
L.~Simon.
\newblock {\em Lectures on Geometric Measure Theory}.
\newblock Proc. Centre Math. Anal. Austral. Nat. Univ., 1983.

\end{thebibliography}
